\documentclass{article}
\usepackage{amssymb,amsmath,amsthm,mathrsfs} 
\usepackage{verbatim}
\usepackage{tikz-cd} 
\usepackage{hyperref}
\usepackage{tabularx}
\usepackage{pgfplots}
\usepackage{pgfplotstable}
\usepackage{filecontents}
\usepackage{geometry}
\usepackage{mathtools}
\usepackage{float}
\usepackage{booktabs}
\usepackage{lipsum}

\let\phi\varphi

\newcommand\tab[1][.5cm]{\hspace*{#1}}
\newtheorem{theorem}{Theorem}[section]
\newtheorem{corollary}[theorem]{Corollary}
\newtheorem{lemma}[theorem]{Lemma}
\newtheorem{proposition}[theorem]{Proposition}

\theoremstyle{definition}
\newtheorem{definition}[theorem]{Definition}
\newtheorem{remark}[theorem]{Remark}

\newtheorem{example}[theorem]{Example}

\newtheorem{conjecture}[theorem]{Conjecture}

\newcommand*{\bbb}[1]{{\mathbb{#1}}}
\newcommand{\C}{\bbb{C}}

\newcommand{\Q}{\bbb{Q}}

\newcommand{\Z}{\bbb{Z}}
\newcommand{\F}{\bbb{F}}

\newcommand{\Gal}{\text{Gal}}

\newcommand{\Nrm}{\text{N}}

\newcommand{\Mod}[1]{\ (\mathrm{mod}\ #1)}

\newcommand{\mO}{\mathcal{O}}

\title{On Gauss factorials and their connection to the cyclotomic $\lambda$-invariants of imaginary quadratic fields.}

\author{Matt Stokes}

\begin{document}

\maketitle

\pagestyle{plain}

\begin{abstract}
In this paper we establish a connection between the Gauss factorials and Iwasawa's cyclotomic $\lambda$-invariant for an imaginary quadratic field $K$.  As a result, we will explain a correspondence between the 1-exceptional primes of Cosgrave and Dilcher \cite{cd1}, \cite{cd} for $m = 3$ and $m = 4$, and the primes for which the $\lambda$-invariants for $K = \Q(\sqrt{-3})$ and $K = \Q(i)$ are greater than one, respectively.  We refer to the latter primes as ``non-trivial'' for their respective fields.  We will also see that similar correspondences are true for $K = \Q(\sqrt{-d})$ when $d = 2,5$ and $6$.  As a corollary we find that primes $p$ of the form $p^2 = 3x^2 + 3x + 1$ are always non-trivial for $K = \Q(\sqrt{-3})$.  Last, we show that the non-trivial primes $p$ for $K = \Q(i)$ and $K = \Q(\sqrt{-3})$ are characterized by modulo $p^2$ congruences involving Euler and Glaisher numbers respectively.
\end{abstract}


\section{Introduction and statement of main results}

Let $p$ be an odd prime, and $d > 0$ a square-free integer.  Denote $K = \Q(\sqrt{-d})$ and $\lambda_p(K)$ to be Iwasawa's $\lambda$-invariant for the cyclotomic $\Z_p$-extension of $K$.  In \cite{dfks}, Dummit, Ford, Kisilevsky and Sands compute $\lambda_p(K)$ for various primes and imaginary quadratic fields.  They define the non-trivial primes of $K$ to be those which satisfy $\lambda_p(K) > 1$ (non-trivial since $\lambda_p(K) >0$ whenever $p$ splits in $K$).  For example, Table 1 gives the non-trivial primes for $K = \Q(\sqrt{-3})$ and $K = \Q(i)$ for primes $p < 10^7$ (see Table 1 in \cite{dfks} for all other imaginary quadratic fields with discriminants up to 1,000).
\begin{table}[h!]\label{table1}
\begin{center}
\caption{Non-trivial primes $p<10^7$ of $ \Q(\sqrt{-3})$ and $\Q(i)$.}
\vskip5mm
\begin{tabular}{|c |c c c c c c|} 
 \hline
$K= \Q(\sqrt{-3})$ & 13 & 181 & 2521 & 76543 & 489061 & 6811741 \\ 
 \hline
 $K = \Q(i)$ & 29789 &  & & & &  \\
 \hline
\end{tabular}
\end{center}
\end{table}

Authors such as Ellenberg, Jain, and Venkatesh \cite{ejv}, Horie \cite{horie}, Ito \cite{ito}, and Sands \cite{sands} have studied $\lambda_p(K)$ by fixing a prime $p$ and varying the imaginary quadratic field $K$.  Dummit, Ford, Kisilevsky, and Sands \cite{dfks}, and Gold \cite{gold} have studied the case when $K$ is fixed and $p$ varies (which is the point of view we take in this paper), but less seems to be known in this situation.  Another point of view might be to fix both $p$ and $K$ and vary the $\Z_p$-extension of $K$.  Interestingly, Sands \cite{sands1} has shown that if $p$ does not divide the class number of $K$, and the cyclotomic $\lambda$-invariant $\lambda_p(K) \leq 2$, then every other $\Z_p$-extension $K_{\infty}/K$ has $\lambda_p \leq 2$ and $\mu_p = 0$.  Therefore, knowing the non-trivial primes $p$ of $K$ is important for our overall understanding of the other $\Z_p$-extensions of $K$.

On the other hand, for $m \in \Z^+$ we have the seemingly unrelated 1-exceptional primes $p$ for $m$ studied by Cosgrave and Dilcher, that is, primes $p \equiv 1 \Mod m$ such that $\left(\frac{p^2-1}{m}\right)_p^{p-1}! = \left( \prod_{a = 1\atop \gcd(a,p) = 1}^{\frac{p^2-1}{m}} a \right)^{p-1} \equiv 1 \Mod{p^2}$.  Surprisingly, the primes $p$ in Table 1 are exactly the 1-exceptional primes for $m= 3$ and $m = 4$ respectively, with $p < 10^7$ (see the next section, or \cite{cd1} and \cite{cd} to learn about 1-exceptional primes).  

The main result of this paper is Theorem \ref{gaussthrm}, which is a criterion in terms of Gauss factorials that give $\lambda_p(K) > 1$ (this is a condition that works for every imaginary quadratic field $K$ and any primes $p$ that split in $K$).  From this, we obtain an explanation for the apparent connection between the $1$-exceptional primes for $m= 3$ and $m=4$, and the non-trivial primes of $K = \Q(\sqrt{-3})$ and $K = \Q(i)$, as well as some similar results for $K = \Q(\sqrt{-d})$ with $d = 2,5$ and $6$:

\begin{theorem}\label{mainthrm}
Let $K =\Q(\sqrt{-d})$ and $D = 2d$ if $d \equiv 3 \Mod4$ and $D = 4d$ otherwise.  Let $r\in \Z^+$ such that $p^r \equiv 1 \Mod{D}$, and suppose that $p$ does not divide the class number of $K$.  Then for $d = 1,2,3,5$ and $6$ we have
\[
\lambda_p(K) > 1 \iff  \left(\frac{\left( \frac{p^{2r} -1}{D}\right)^2_p!}{\left( \frac{p^{2r}-1}{D/2}\right)_p!}\right)^{p-1} \equiv 1 \Mod{p^2}.
\]
In particular, $p$ is $1$-exceptional for $m = 3$ if and only if $\lambda_p(\Q(\sqrt{-3}) > 1$ and $p$ is $1$-exceptional for $m = 4$ if and only if $\lambda_p(\Q(i)) > 1$.
 \end{theorem}
 
 The proof of Theorem \ref{mainthrm} relies on the fact the the fields $K =\Q(\sqrt{-d})$, where $d = 1,2,3,5$ and $6$, have so called \textit{``maximal class numbers''} (see Definition \ref{maximal}).  We will prove Theorem \ref{maxclassnumber} which tells us that these are the only imaginary quadratic fields with such class numbers, under the assumption that the generalized Riemann hypothesis is true.

As a corollary of Theorem \ref{mainthrm} we will see that primes $p$ of the form $p^2 = 3x^2 + 3x + 1$ with $x \in \Z$ always give $\lambda_p(\sqrt{-3}) > 1$.  However, the converse does not hold (see Remark \ref{rempell}).  Theorem \ref{mainthrm} also leads to
\begin{corollary}\label{bernoullicor}
For $K = \Q(\sqrt{-d})$ for $d = 1, 2, 3, 5$ and $6$, we have
\[
\lambda_p(K) > 1 \iff B_p(2/D) \equiv 2^pB_p(1/D) \Mod{p^3}
\]
where $B_n(x)$ is the $n$-th Bernoulli polynomial.
\end{corollary}
In particular, we obtain some interesting conditions for the non-trivial primes of $K = \Q(i)$ and $K = \Q(\sqrt{-3})$ in terms of Glaisher and Euler numbers respectively.  Recall the Euler numbers $\{E_n\}$ and Glaisher numbers $\{G_n\}$ are defined by 
\[
\sum_{n = 0}^{\infty}E_n\frac{x^n}{n!} = \frac{2}{e^x + e^{-x}} \tab \text{ and } \tab \sum_{n = 0}^{\infty} G_n \frac{x^n}{n!} = \frac{3/2}{e^{x} + e^{-x} + 1}.
\]
We will prove:

\begin{corollary}\label{cor1.2}
Let $p \equiv 1 \Mod 4$ be a prime and $E_n$ denote the $n$-th Euler number.  Then $\lambda_p(\Q(i)) > 1$ if and only if $E_{p-1} \equiv 0 \Mod{p^2}$.
\end{corollary}

\begin{corollary}\label{cor1.3}
Let $p \equiv 1 \Mod 3$ be a prime and $G_n$ denote the $n$-th Glaisher number.  Then $\lambda_p(\Q(\sqrt{-3})) > 1$ if and only if $G_{p-1} \equiv 0 \Mod{p^2}$.
\end{corollary}

\begin{remark}
The numbers $\{G_n\}$ were studied by Glaisher in \cite{glaisher1} and \cite{glaisher2} in which they are referred to as \textit{$I$-numbers}.
\end{remark}

An analogue of Theorem \ref{gaussthrm} for primes $p$ giving $\lambda_p(\Q(\sqrt{-d}))> 2$ is proved in the author's PhD thesis, but uses a different technique involving $p$-adic $L$-functions.


\section{Gauss factorials and exceptional primes}
In this section we define Gauss factorials and exceptional primes, as well as state some results that will be needed for the proof of Theorem \ref{mainthrm} as well as Corollary \ref{corollary1}.  For $N, n \in \Z^+$ the Gauss factorial of $N$ with respect to $n$ is defined as
\[
N_n! \hskip1mm= \hskip2mm \prod_{\mathclap{\substack{i = 1 \\ \gcd(i,n) = 1}}}^{N} i
\]
In \cite{cd}, Cosgrave and Dilcher investigate multiplicative orders modulo powers of $p$ of the following Gauss factorials
\[
\left( \frac{p^{\alpha} - 1}{m}\right)_p! 
\] 
where $m,\alpha \in \Z^+$, with $m$ and $\alpha$ greater than 2, and $p \equiv 1 \Mod m$.  If $\gamma_{\alpha+1}^m(p)$ is the multiplicative order of $\left( \frac{p^{\alpha+1}-1}{m}\right)_p!$ modulo $p^{\alpha+1}$, then Cosgrave and Dilcher define $p$ to be $\alpha$-exceptional for $m$ if $\gamma_{\alpha+1}^m(p)$ and $\gamma_{\alpha}^m(p)$ are the same modulo a factor of $2^{\pm1}$ (otherwise  $\gamma_{\alpha+1}^m(p) = p\gamma_{\alpha}^m(p)$ or  $\gamma_{\alpha+1}^m(p) = 2^{\pm1}p\gamma_{\alpha}^m(p)$, see Theorem 1 and Definition 1 in \cite{cd}).  Further, Theorem 3 in \cite{cd} shows that if $p$ is $\alpha$ exceptional for $m$, then $p$ is also $(\alpha-1)$-exceptional for $m$.  For our purposes, we will not need this much precision on the multiplicative orders, and we will instead use the equivalent definition:

\begin{definition}\label{exceptional}
For $\alpha \in \Z^+$, we say that $p$ is $\alpha$-exceptional for $m$ if and only if $\left(\frac{p^{\alpha+1} -1}{m}\right)^{{p-1}}_p! \equiv 1 \Mod{p^{\alpha+1}}$.  
\end{definition}

We contrast Definition \ref{exceptional} with the following definition of ``non-trivial'' primes.  Theorem \ref{mainthrm} will show that the primes in each of the two definitions are the same when $m = 3$ and $K = \Q(\sqrt{-3})$, and when $m = 4$ and $K = \Q(i)$:

\begin{definition}\label{nontrivial}
Given an imaginary quadratic field $K$, we say that $p$ is non-trivial for $K$ if $\lambda_p(K) > 1$.
\end{definition}

\begin{example}
Let $p = 13$ and $m = 3$.  Then $\gamma_{1}^3(13) = 12$, $\gamma_2^3(13) = 12$, $\gamma_3^3(13) = 12\cdot13$, $\gamma_4^3(13) = 12\cdot13^2$, $\gamma_5^3(13) = 12\cdot13^3$ and so on (``and so on" since Theorem 3 in \cite{cd} says that $p$ is $(\alpha+1)$-exceptional for $m$ implies $p$ is also $\alpha$-exceptional for $m$).  The next few values of $p$ such that $\gamma_1^3(p) = \gamma_2^3(p)$ are $p = 181, 2521, 76543$ and so on.  On the other hand, if $m =4$ and $p = 29789$, then $\gamma_2^4(p) = \frac{1}{2}\gamma_1^4(p)$, and is the only known such example for $p < 10^{11}$ (see also Table 1 in \cite{cd1} for $\gamma_{\alpha}^4(p)$ with $1\leq \alpha \leq 5$, $p\leq 37$ and $p \equiv 1 \Mod 4$).
\end{example}

The following results of Cosgrave and Dilcher will be important later on:

\begin{theorem}[Cosgrave-Dilcher \cite{cd}]\label{thrm3}
Let $p \equiv 1 \Mod 6$ be a prime. Then $p$ is 1-exceptional for $m =3$ if and only if $p$ is 1-exceptional for $m =6$.  
\end{theorem}

\begin{theorem}[Cosgrave-Dilcher \cite{cd}]\label{thrm4}
Let $p \equiv 1 \Mod 6$ be a prime and $n \in \Z^+$. Then 
\[
\left(\left(\frac{p^n -1}{3}\right)_p!\right)^{24} \equiv \left(\left(\frac{p^n-1}{6}\right)_p!\right)^{12} \Mod{p^n}.
\]

\end{theorem}

\begin{theorem}[Cosgrave-Dilcher \cite{cd1}]\label{thrm5}
Every prime $p \equiv 1 \Mod 6$ that satisfies $p^2 = 3x^2 + 3x + 1$ for some $x \in \Z$ is $1$-exceptional for $m=3$.  Equivalently, if $\gamma = 2+\sqrt{3}$ and $q \in \Z^+$, then any prime of the form
\[
p = \frac{\gamma^q + \gamma^{-q}}{4}
\]
is $1$-exceptional for $m=3$.  
\end{theorem}

\begin{definition}
We shall refer to the primes $p \equiv 1 \Mod 3$ such that $p^2 = 3x^2 + 3x + 1$ for some $x \in \Z$ as Cosgrave-Dilcher primes.
\end{definition}

\begin{remark}\label{rempell} In \cite{cd1} and \cite{cd} Cosgrave and Dilcher rearranged the equation $p^2 = 3x^2 + 3x + 1$ into $(2p)^2 - 3(2x + 1)^2 = 1$, which can be viewed as the Pell equation $X^2 - 3Y^2 = 1$.  It is from the theory of these equations that we obtain the primes $p = (\gamma^q + \gamma^{-q})/4$.  Also, $q$ is necessarily prime (see lemma 7 in \cite{cd1}).  It should be mentioned that the converse of Theorem \ref{thrm5} does not hold.  For example, $p = 76543$ is 1-exceptional for 3 but is not a Cosgrave-Dilcher prime ($p=76543$ is the only such example for $p<10^{12}$).  It is unknown whether or not there are infinitely many Cosgrave-Dilcher primes, and the question seems to be analogous to that of the infinitude of Fibonacci primes.  In a moment we will list some new Cosgrave-Dilcher primes (see Example \ref{cdremark}).
\end{remark}


\section{Proof of main Theorems}

In this section we will prove Theorem \ref{mainthrm} from which we immediately obtain as a Corollary:

\begin{corollary}\label{corollary1}
Let $p \equiv 1 \Mod 6$ be a Cosgrave-Dilcher prime.  Then $\lambda_p(\Q(\sqrt{-3})) >1$.
\end{corollary}

\begin{example}\label{cdremark}
Using Corollary \ref{corollary1} we may add to the non-trivial primes of $\Q(\sqrt{-3})$ in Table \ref{table1} by searching for Cosgrave-Dilcher primes.  The following table contains $p = (\gamma^q + \gamma^{-q})/4$ with $q \leq 79$:
\begin{center}
\begin{tabular}{ | m{3em} | m{10cm}|}
 \hline
 $q=3$ & $p = 13 $ \\
 \hline
 $q=5$ &  $p = 181$  \\
 \hline
$q=7$ & $p= 2521$    \\
 \hline
 $q=11$ &  $p = 489061$  \\
\hline
 $q=13$ & $p = 6811741$  \\
 \hline
$q=17 $ &  $p = 1321442641$  \\
 \hline
 $q=19$  &  $p = 18405321661$  \\
 \hline
 $q=79$ & $p = 381765135195632792959100810331957408101589361$  \\
\hline
\end{tabular}
\end{center}
One may further verify using any standard CAS that the primes $79< q \leq 10,000$ giving $1$-exceptional primes $p  = (\gamma^q + \gamma^{-q})/4$ for $m = 3$ (and therefore non-trivial primes of $\Q(\sqrt{-3})$) are $q = 151$, $199$, $233$, $251$, $317$, $863$, $971$, and $q = 3049$, $7451$, and $7487$ giving probable primes $p$ (the non-trivial probable prime corresponding to $q = 7487$ is 4282 digits long). 
\end{example}

 Let $d$ be a square-free integer, $K = \Q(\sqrt{-d})$, $D = 2d$ if $d \equiv 3 \Mod 4$ and $D = 4d$ otherwise.  Let $p>2$ be a prime such that $p\equiv 1 \Mod D$ with $p\mO_K = \frak{p}\bar{\frak{p}}$ and $\mathcal{P}$ be a prime in $\Q(\zeta_{D})$ above $\frak{p}$, where $\zeta_D$ is a primitive $D$-th root of unity.  Let $\bar{\mathcal{P}}$ be the complex conjugate of $\mathcal{P}$.  Denote $G = \Gal(\Q(\zeta_{D})/\Q)$ and $\chi_K = \chi$ to be the imaginary quadratic character for $K$.  We have for $x \in \Q(\zeta_D)$
\[
\Nrm_{\Q(\zeta_{D})/K}(x) = \prod_{\mathclap{i = 1 \atop {\chi(i) = 1 \atop \gcd(i,D) = 1}}}^{D} \sigma_{i}(x) \in K
\]
where $\sigma_i \in G$ acts by $\sigma_i (\zeta_{D} )=  \zeta_D^{i}$.  We will also denote $\mathcal{P}_i = \sigma_i(\mathcal{P})$ so that $\Nrm_{\Q(\zeta_D)/K}(\mathcal{P}_i) = \frak{p}$.  We will now work towards proving the following result from which Theorem \ref{mainthrm} will follow. 

\begin{theorem}\label{gaussthrm}
Let $K = \Q(\sqrt{-d})$ be any imaginary quadratic field, and $D$ be as above.  Let $p$ be a prime and $r \in \Z^+$ such that $p^r \equiv 1 \Mod D$, $p \nmid h_K$, and $p \neq 3$ whenever $\chi_K(2) = -1$ and $K \neq \Q(\sqrt{-3})$.  Then,
\[
\lambda_p(K) > 1 \iff  \left(\prod^{D}_{i = D/2\atop {\chi(i) = 1 \atop \gcd(i, D) = 1}} \frac{\left((D-i) \frac{p^{2r} -1}{D}\right)^2_p!}{\left((D-i) \frac{p^{2r}-1}{D/2}\right)_p!} \prod^{D/2}_{i = 1\atop {\chi(i) = 1 \atop \gcd(i, D) = 1}} \frac{\left(i \frac{p^{2r} -1}{D/2}\right)_p!}{\left(i \frac{p^{2r} -1}{D}\right)^2_p!}\right)^{p-1} \equiv 1 \Mod{p^2}.
\]
\end{theorem}

The first step of the proof is to write $\bar{\frak{p}}$ in terms of Jacobi sums.  Consider the multiplicative character $\psi : \mO_{\Q(\zeta_D)}/ \mathcal{P} \to \C^{\times}$ of order $D$ modulo $\mathcal{P}$.  We denote 
\[
J(\psi) = \sum_{a \in \F_{p}} \psi(a) \psi(1-a)
\]
to be the Jacobi sum for $\psi$.  Denote $0 \leq L(j) < D$ to be reduction of $j$ modulo $D$, and  for $1 \leq i < D/2$ we define 
\[
S_i(D) = \{ j \,:\, 0< j < D; \,\gcd(j,D) = 1;\, L(ji) <  D/2 \}.
\]
Then from Theorem 2.1.14 in \cite{bew} we have 
\[
J(\psi^i)\mO_{\Q(\zeta_{D})} = \prod_{j \in S_i(D)} \mathcal{P}_{j^{-1}}.
\]

\begin{proposition}\label{propfrakp}
Denote $h_K = h$ to be the class number for $K = \Q(\sqrt{-d})$.  With the notation fixed above, we have
\[
\bar{\frak{p}}^{t}= \left(\hskip6mm \prod^{D}_{\mathclap{i = D/2\atop {\chi(i) = 1 \atop \gcd(i, D) = 1}}} J(\psi^{ i}) \middle / \prod^{D/2}_{\mathclap{i = 1\atop{ \chi(i) = 1 \atop \gcd(i,D) = 1}}} J(\psi^{-i})\right) \mO_{\Q(\zeta_{D})}
\]
where $t = \pm h(2 - \chi(2))$ if $d \neq 1$ or $3$, else $t = \pm 1$.  The sign of $t$ depends on the number of quadratic residues modulo $D$ between $1$ and $D/2$.
\end{proposition}

\begin{proof}
Denote 
\[
a^+ =\# \{ 0< j < D/2\:\, \gcd(j,D) = 1, \, \chi(j) = 1\} 
\]
\[
 a^- =\# \{ 0< j < D/2\:\, \gcd(j,D) = 1, \, \chi(j) = - 1\}.
\]
It is well known that $\pm h = (a^+ - a^-)/(2 - \chi(2))$ when $d$ is not $1$ or $3$ (it is easy to see what happens in those cases, so we will assume $d > 3$).  If $N$ is the norm from $\Q(\zeta_D)$ to $K$ then 
\[
N(J(\psi^{-1}))\mO_{\Q(\zeta_{D})} = \prod_{j = 1 \atop \gcd(j,D) =1}^{D/2} N(\mathcal{\bar{P}}_{j^{-1}}) = \frak{p}^{a^-} \bar{\frak{p}}^{a^+}
\]
and also $J(\psi^i)J(\psi^{-i}) = p$.  Then the ideal
\begin{align*}
\left(\hskip 6mm \prod^{D}_{\mathclap{i = D/2\atop {\chi(i) = 1 \atop \gcd(i, D) = 1}}} J(\psi^{ i}) \middle / \prod^{D/2}_{\mathclap{i = 1\atop{ \chi(i) = 1 \atop \gcd(i,D) = 1}}} J(\psi^{-i}) \right) & =  \left(\hskip 6mm \prod^{D}_{\mathclap{i = D/2\atop {\chi(i) = 1 \atop \gcd(i, D) = 1}}} J(\psi^{ i}) \prod^{D}_{\mathclap{i = D/2\atop {\chi(i) = 1 \atop \gcd(i, D) = 1}}} J(\psi^{-i}) \middle / \prod^{D/2}_{\mathclap{i = 1\atop{ \chi(i) = 1 \atop \gcd(i,D) = 1}}} J(\psi^i) \prod^{D}_{\mathclap{i = D/2\atop {\chi(i) = 1 \atop \gcd(i, 2m) = 1}}} J(\psi^{-i})\right)
\\
&= \left(\hskip6mm \prod_{\mathclap{i = D/2\atop {\chi(i) = 1 \atop \gcd(i, D) = 1}}}^{D} p \middle / N(J(\psi^{-1})) \right) = \frac{(\frak{p}\bar{\frak{p}})^{ a^-}}{\frak{p}^{a^-}\bar{\frak{p}}^{a^+}} =\bar{ \frak{p}}^{\pm h(2 - \chi(2))}.
\end{align*}
\end{proof}

Theorem \ref{gaussthrm} will now follow from Gold's criterion:

\begin{theorem}[Gold's criterion (Theorem 4 in \cite{gold})] \label{gold}
Let $K$ be an imaginary quadratic field, and $p > 2$ be a prime such that $p$ does not divide the class number $h_K$ of $K$.  
\begin{itemize}
\item[i.] If $p$ splits in $K$ then $\lambda_p(K) > 0$.
\item[ii.] Suppose $p\mO_K = \frak{p}\bar{\frak{p}}$ and write $\frak{p}^{h_K} = (\alpha)$. Then $\lambda_p(K) > 1$ if and only if $\alpha^{p-1} \equiv 1 \Mod{\bar{\frak{p}}^2}$.
\end{itemize}
\end{theorem}

\begin{proof}[Proof of Theorem \ref{gaussthrm}]
Let $r \in \Z^+$ such that $p^r \equiv 1 \Mod{D}$.  Working inside the localization $K_{\frak{p}} \cong \Q_p$, we have  $J(\psi^{-i}) \equiv \frac{\left(i \frac{p^{2r} -1}{D/2}\right)_p!}{\left(i \frac{p^{2r} -1}{D}\right)^2_p!} \Mod{p^2\Z_p}$ from (9.3.6) in \cite{bew} (which is essentially the Gross-Koblitz formula).
The result now follows from Proposition \ref{propfrakp} and Gold's criterion \ref{gold}.
\end{proof}

We will see that the condition in Theorem \ref{gaussthrm} becomes more compact for a certain family of imaginary quadratic fields.

\begin{definition}\label{maximal}
Let $\chi_K = \chi$ be the imaginary quadratic character for $K$ and $D$ be as above.  We say that $K$ has maximal class number if $\chi_K(i) = 1$ for each $i$ co-prime to $D$ and $1 \leq i \leq D/2$.
\end{definition}

If $h_K$ is the class number for $K \neq \Q(i)$ or $ \Q(\sqrt{-3})$, we have that $(2 - \chi(2))h_K = \left|\sum_{i = 1}^{D/2} \chi(i)\right|$.  Then $\chi(i) = 1$ for each $i$ co-prime to $D$ and $1 \leq i \leq D/2$ if and only if $h_K = \phi(D)/2(2 - \chi(2))$.  

\begin{theorem}\label{maximalthrm}
Suppose $r \in \Z^+$ such that $p^r \equiv 1 \Mod{D}$ and all other notation is as above.  If $K$ has maximal class number, and $p\nmid h_K$, then
\begin{align*}
\lambda_p(K) > 1 \iff  \left(\frac{\left( \frac{p^{2r} -1}{D}\right)^2_p!}{\left( \frac{p^{2r}-1}{D/2}\right)_p!}\right)^{p-1} \equiv 1 \Mod{p^2}.
\end{align*}
\end{theorem}

\begin{proof}
Here we will view $K \subseteq K_{\frak{p}} \cong \Q_p$.  If $K$ has maximal class number, then $S_1(D)$ accounts for all of the quadratic residues between $1$ and $D/2$, and so $J(\psi^{-1}) \in N(\bar{\mathcal{P}}) =\bar{ \frak{p}}$.  Therefore, if $\bar{\frak{p}}^{h_K} = (\alpha)$ for some $\alpha \in K$, we have $J(\psi^{-1})^{h_K} \equiv \alpha u \Mod{p^2\Z_p}$ where $u \in \mO_K^{\times}$.  Now, since $p \nmid h_K$ we have $J(\psi^{-1})^{h_K(p-1)} \equiv 1 \Mod{p^2\Z_p}$ if and only if  $J(\psi^{-1})^{(p-1)} \equiv 1 \Mod{p^2\Z_p}$.  The result now follows from Gold's criterion and the fact that $u^{p-1} =1$.
\end{proof}

\begin{remark}
When $D = 6$, the combination of Theorems \ref{thrm3} and \ref{thrm4} imply that $\lambda_p(\Q(\sqrt{-3})) > 1$ if and only if $p$ is $1$-exceptional for $m = 3$.  When $D = 4$, we have that $\left(\frac{p^2-1}{2}\right)_p^{p-1}! \equiv 1 \Mod{p^2}$ (a corollary of Wilson's theorem), so $\lambda_p(\Q(i)) > 1$ if and only if $p$ is $1$-exceptional for $m = 4$.
\end{remark}

Theorem \ref{mainthrm} now follows as a special case of Theorem \ref{maximalthrm}.  Computations show that $K = \Q(i)$, $\Q(\sqrt{-2})$, $\Q(\sqrt{-3})$, $\Q(\sqrt{-5})$ and $\Q(\sqrt{-6})$ are the only imaginary quadratic fields $K = \Q(\sqrt{-d})$ with $d < 10,000$ having maximal class number.  In fact,

\begin{theorem}\label{maxclassnumber}
Assuming the generalized Riemann hypothesis (GRH) holds for every non-principle primitive imaginary quadratic character, the only imaginary quadratic fields with maximal class number are  $K = \Q(i)$, $\Q(\sqrt{-2})$, $\Q(\sqrt{-3})$, $\Q(\sqrt{-5})$ and $\Q(\sqrt{-6})$.
\end{theorem}

\begin{proof}
Let $d > 0$ be a square free integer and let $D$ and $\chi_D = \chi$ be as above.  Denote $K = \Q(\sqrt{-d})$ and $h_K$ to be the class number of $K$, and assume that $h_K = \phi(D)/2(2-\chi(2))$ (i.e. $h_K$ is maximal).  From Theorem 15 in \cite{JBRLS}, we have 
\[
\phi(D) > \frac{D}{e^{\gamma}\log\log(D) + \frac{3}{\log\log(D)}}
\]
where $e =$ exp$(1)$ and $\gamma = 0.577215665...$ is Euler's constant.  On the other hand, under the assumption of the generalized Riemann hypothesis, Littlewood \cite{littlewood} gave the inequality $h_K < ce^{\gamma}\log\log(D) \sqrt{D}$, where $c$ is an absolute constant.  Recently, this bound has been improved (see \cite{LLS} and \cite{LATTS}) to
\[
h_K \leq \frac{2 e^{\gamma}}{\pi}\sqrt{D}\left(\log\log(D) - \log(2) + \frac{1}{2} + \frac{1}{\log\log(D)}\right)
\]
for $D \geq 5$, and assuming GRH holds.  Thus, when $h_K$ is maximal and $D \geq 5$, the two inequalities above imply
\[
\sqrt{D} < \frac{12 e^{\gamma}}{\pi}\left( e^{\gamma}(\log\log(D))^2 + \frac{3}{(\log\log(D))^2} + e^{\gamma} + 3 \right) < 14(\log\log(D))^2 + 140.
\]
This inequality does not hold for long.  Indeed, set $f(x) =  \sqrt{x} - 14(\log\log(x))^2 + 140$ and notice that $f'(x) = \frac{1}{2\sqrt{x}} - \frac{28\log\log(x)}{x\log(x)} > 0$ precisely when $x\log(x) > 56\sqrt{x} \log\log(x)$, which will eventually hold for all $x$ sufficiently large (e.g. for all $x > 300$).  Therefore, we have that $f(x)$ is strictly increasing on $[300, \infty)$.  We also have that $f(300) > 0$, so the inequality $ \sqrt{D} > 14(\log\log(D))^2 + 140$ holds for all $D > 300$.  Therefore, there are no imaginary quadratic fields with $D> 300$ having maximal class number.  It is easy to check that the only imaginary quadratic fields with $D \leq 300$ and maximal class number are the ones listed above.
\end{proof}

\section{Proof of Corollaries}

We now turn to the proofs of Corollaries \ref{bernoullicor}, \ref{cor1.2} and \ref{cor1.3} for which we will need some preliminary results.  For $a$ co-prime to $p$ the Fermat quotient is defined as $q_p(a) = (a^{p-1}-1)/p$, which is an integer by Fermat's little Theorem.  The Fermat quotient has logarithmic properties, that is, for $a$ and $b$ co-prime to $p$,
\[
q_p(a) + q_p(b) \equiv q_p(ab) \Mod p \tab \text{and} \tab q_p(a) - q_p(b) \equiv q_p\left( a/b  \right) \Mod p
\]
as well as
\[
q_p(a +p) \equiv q_p(a) - \frac{1}{a} \Mod p.
\]
Denote $H_n = \sum_{a = 1}^n 1/a$ to be the $n$-th harmonic number and $w_p = ((p-1)! + 1)/p$ to be the Wilson quotient (also an integer by Wilson's Theorem).  It is well known that $w_p \equiv \sum_{a = 1}^{p-1}q_p(a) \Mod p$.  
\begin{lemma}\label{lem1.5}
Let $p>2$ be a prime.  For any $b \in (\Z/p^2\Z)^{\times}$ such that $b = b_0 + b_1p$ with $1 \leq b_0 \leq p-1$ and $0 \leq b_1 \leq p-1$, we can write $b \equiv b_0^p \left( 1 +\left( \frac{b_1}{b_0} - q_p(b_0)  \right)p \right) \Mod{p^2}$.
\end{lemma}

\begin{proof}
Let $b \in (\Z/p^2\Z)^{\times}$ such that $b = b_0 + b_1p$ with $1 \leq b_0 \leq p-1$ and $0 \leq b_1 \leq p-1$.  Then setting $x = b_1/b_0$, we see that $1 + px \equiv (pq_p(b_0)+1)(1+px-pq_p(b_0)) \Mod{p^2}$.  Since $b_0^{p-1} = 1 + pq_p(b_0)$, we obtain the result by multiplying through by $b_0$.
\end{proof}

\begin{proposition}\label{prop1.1}
Suppose $m \in \Z$ with $m \geq 2$ and $p \equiv 1 \Mod m$ is a prime.  Then
\[
\left( \frac{p^2 -1}{m}\right)_p^{p-1}!  \equiv 1 \Mod{p^2} \iff \frac{1}{m}(w_p- H_{\frac{p-1}{m}} ) -\sum_{a = 1}^{\frac{p-1}{m}}q_p(a) \equiv 0 \Mod p.
\]
\end{proposition}

\begin{proof}
Using Lemma \ref{lem1.5}, we have
\begin{align*}
\left( \frac{p^2 -1}{m}\right)_p^{p-1}! & = \prod_{a = 1\atop \gcd(a,p)=1}^{\frac{p^2 - 1}{m}} a^{p-1}  =\left( \prod_{a = 1}^{p-1} \prod_{b = 0}^{\frac{p-1}{m}-1} (a + bp)^{p-1} \right) \left( \prod_{a = 1}^{\frac{p-1}{m}} \left(a + \frac{p-1}{m}p \right)^{p-1}\right)
\\
& \equiv \left( \prod_{a = 1}^{p-1} \prod_{b = 0}^{\frac{p-1}{m}-1} \left(1 + \left(\frac{b}{a} - q_p(a)\right) p \right) \right) \left( \prod_{a = 1}^{\frac{p-1}{m}} \left( 1 + \left( \frac{\frac{p-1}{m}}{a} - q_p(a)\right)p \right)\right) \Mod{p^2}
\\
& \equiv  \left( \prod_{a = 1}^{p-1} \prod_{b = 0}^{\frac{p-1}{m}-1} (1 +p)^{\frac{b}{a} - q_p(a)} \right) \left( \prod_{a = 1}^{\frac{p-1}{m}} (1+p)^{\frac{\frac{p-1}{m}}{a} - q_p(a)} \right) \Mod{p^2}.
\end{align*}
Combining all the factors of $(1+p)$ we get the desired sum in the exponent which is taken modulo $p$ (since $1 + p$ is a $p$-th root of unity modulo $p^2$).  It is known that $H_{p-1} \equiv 0 \Mod p$.  Hence, $\sum_{a =1}^{p-1} \sum_{b = 0}^{\frac{p-1}{m}-1}\frac{b}{a} \equiv 0 \Mod p$.  The result now follows.
\end{proof}

Recall that the Bernoulli numbers $\{B_n\}$ and the Bernoulli polynomials $\{B_n(t)\}$ are defined by
\[
\sum_{n = 0}^{\infty}B_n\frac{x^n}{n!} = \frac{x}{e^x -1} \tab \text{and} \tab \sum_{n = 0}^{\infty}B_n(t)\frac{x^n}{n!} = \frac{xe^{xt}}{e^x -1}.
\]

\begin{lemma}\label{computation}
Let $p$ be a prime such that $p\equiv 1 \Mod {2m}$.  Then
\[
\left(\frac{\left(\frac{p^2-1}{m}\right)_p!}{\left(\frac{p^2-1}{2m}\right)^2_p!}\right)^{p-1} \equiv 1 \Mod{p^2} \iff \frac{B_p(1/m) - 2^{p}B_{p}(1/2m)}{p^2} \equiv 0 \Mod{p}.
\]
\end{lemma}

\begin{proof}
For any $n \in \Z^+$ with $p \equiv 1 \Mod{n}$, we use the relation $B_{p}(x+1) - B_{p}(x) = px^{p-1}$ along with the properties of the Fermat quotient to obtain
\[
\sum_{a = 1}^{\frac{p-1}{n}}q_p(a) \equiv \left( \frac{n^{p-1}}{p} \left (\frac{p^2}{n}B_{p-1}- B_{p} (1/n)\right) - \frac{p-1}{n} \right)+\frac{1}{n}q_p(n) - \frac{1}{n}H_{\frac{p-1}{n}}  \Mod{p}.
\]
Then for $p \equiv 1 \Mod {2m}$, a straightforward computation gives
\[
\sum_{a = 1}^{\frac{p-1}{m}}q_p(a) - 2\sum_{a = 1}^{\frac{p-1}{2m}}q_p(a) \equiv -\frac{B_p(1/m) - 2^{p}B_{p}(1/2m)}{p^2}  - \frac{1}{m}H_{\frac{p-1}{m}} + \frac{1}{m}H_{\frac{p-1}{2m}} \Mod p.
\]
From Proposition \ref{prop1.1} we know that $\left(\frac{\left(\frac{p^2-1}{m}\right)_p!}{\left(\frac{p^2-1}{2m}\right)^2_p!}\right)^{p-1} \equiv (1 + p)^{\xi} \Mod{p^2}$, where
\begin{align*}
\xi &= \frac{1}{m}(w_p- H_{\frac{p-1}{m}} ) -\sum_{a = 1}^{\frac{p-1}{m}}q_p(a) - 2\left(\frac{1}{2m}(w_p- H_{\frac{p-1}{2m}} ) -\sum_{a = 1}^{\frac{p-1}{2m}}q_p(a)\right)
\\
&\equiv - \frac{B_p(1/m) - 2^{p}B_{p}(1/2m)}{p^2} \Mod p.
\end{align*}
The result now follows.
\end{proof}

Corollary \ref{bernoullicor}] is an immediate consequence of Lemma \ref{computation}.  We also have,


\begin{proof}[Proof of Corollary \ref{cor1.2}]
Let $p \equiv 1 \Mod 4$.  We have seen  from Lemma \ref{computation} that $p$ is 1-exceptional for $4$ if and only if $B_p(1/2) - 2^{p}B_{p}(1/4) \equiv 0 \Mod {p^3}$.  But from \cite{lehmer} we know that $B_p(1/2) = 0$ and $B_p(1/4) = -pE_{p-1}/4^p$.  Corollary \ref{cor1.2} now follows from Theorem \ref{mainthrm}.
\end{proof}

\begin{remark}
The proof also shows that $E_{p-1} \equiv 0 \Mod p$ when $p \equiv 1 \Mod 4$, although this was already observed by Zhang in \cite{zhang}.
\end{remark}


The proof of Corollary \ref{cor1.3} will be similar to that of Corollary \ref{cor1.2}, but will instead involve the Glaisher numbers $\{G_n\}$.  Since these numbers are less well known we will take a moment to view some of their properties.  In particular, we will see that for odd $n \geq 1$, $B_n(1/3) = -(n+1)G_{n-1}/3^{n-1}$.  Recall the Glaisher numbers $\{G_n\}$ are defined by 
\[
\frac{3/2}{e^x + e^{-x} + 1} = \sum_{n = 0}^{\infty} G_n \frac{x^n}{n!}.
\]
Notice that $2 \sum_{n = 0}^{\infty} G_{2n+1} \frac{x^{2n+1}}{(2n+1)!} =  \sum_{n = 0}^{\infty} G_n \frac{x^n}{n!} -  \sum_{n = 0}^{\infty} G_n \frac{(-x)^n}{n!}
 = 0$ so that $G_n = 0$ whenever $n$ is odd, and $ \sum_{n = 0}^{\infty} G_n \frac{x^n}{n!} =  \sum_{n = 0}^{\infty} G_{2n} \frac{x^{2n}}{(2n)!}$.  We also know from \cite{glaisher2} that $G_n$ can only have powers of 3 in the denominator.  
\begin{example}
In the following table we list all primes $p \equiv 1 \Mod 3$ and $7 \leq p \leq 193$ in the first column, along with the reduced values of $G_{p-1} \Mod p$ in the second column and $G_{p-1} \Mod{p^2}$ in the third column:

\begin{center}
\begin{tabularx}{0.36\textwidth} { 
  | >{\raggedright\arraybackslash}X 
  | >{\centering\arraybackslash}X 
  | >{\centering\arraybackslash}X |}
 \hline
7  & 0  &  42 \\
 \hline
13  &  0 & 0  \\
 \hline
19  & 0  & 342  \\
 \hline
31  &  0 &  434 \\
 \hline
37  & 0  &  1332 \\
 \hline
43  &  0 & 559  \\
 \hline
61  & 0  &  3660 \\
 \hline
67  &  0 & 3685  \\
 \hline
73  & 0  &  803  \\
\hline
79  &  0 & 2844  \\
\hline
\end{tabularx}
\quad
\begin{tabularx}{0.36\textwidth} { 
  | >{\raggedright\arraybackslash}X 
  | >{\centering\arraybackslash}X 
  | >{\centering\arraybackslash}X |}
 \hline
97  &  0 & 1940  \\
 \hline
103  & 0  & 1133  \\
 \hline
109  & 0  &  7521  \\
\hline
127  & 0  &  16002 \\
 \hline
139 &  0 & 5282  \\
 \hline
151  & 0  & 15855  \\
 \hline
157  &  0 & 785  \\
 \hline
163  &  0 &  24939 \\
\hline
181  & 0  & 0  \\
\hline
193 & 0 & 26441\\
\hline
\end{tabularx}
\end{center}
Notice that $13$ and $181$ are the first two $1$-exceptional primes for $m = 3$.  It also appears that $G_{p-1} \equiv 0 \Mod p$ for all $p \equiv 1 \Mod 3$, which we will soon see is true.
\end{example}

We will now show that $B_n(1/3) = -(n+1)G_{n-1}/3^{n-1}$ for odd $n \geq 1$.  It should be noted that this result is already known (see page 352 in \cite{lehmer}), but not commonly stated or proven in the literature.  Observe that
 \begin{align*}
 \frac{-x}{e^{\frac{1}{3}x}+e^{-\frac{1}{3}x}+1} &= -\frac{2}{3}x \frac{3/2}{e^{\frac{1}{3}x}+e^{-\frac{1}{3}x}+1}
 = -\frac{2}{3}x \sum_{n = 0}^{\infty} G_{2n} \frac{\left( \frac{1}{3}x\right)^{2n}}{(2n)!}
  \\
  & = 2\sum_{n = 0}^{\infty} -\frac{(2n+1)G_{2n}}{3^{2n+1}} \frac{x^{2n+1}}{(2n+1)!}
 \end{align*}
 and at the same time
 \begin{align*}
2\sum_{n = 0}^{\infty} B_{2n+1}(1/3) \frac{x^{2n+1}}{(2n+1)!}  =  \frac{x(e^{\frac{1}{3}x}-e^{\frac{2}{3}x})}{e^x -1} = \frac{xe^{\frac{1}{3}x}(1-e^{\frac{1}{3}x})}{(e^{\frac{1}{3}x} -1)(e^{\frac{2}{3}x}+e^{\frac{1}{3}x}+1)} = \frac{-x}{e^{\frac{1}{3}x}+e^{-\frac{1}{3}x}+1} 
\end{align*}
 Therefore,
 \[
 \sum_{n = 0}^{\infty} -\frac{(2n+1)G_{2n}}{3^{2n+1}} \frac{x^{2n+1}}{(2n+1)!} = \sum_{n = 0}^{\infty} B_{2n+1}(1/3) \frac{x^{2n+1}}{(2n+1)!}
 \]
 which implies, 
 \[
 B_{2n+1}(1/3) =  -\frac{(2n+1)G_{2n}}{3^{2n+1}}.
 \]
 For $k, n \in \Z^+$, we also have Raabe's multiplication formula $B_n(kx) = k^{n-1}\sum_{j=0}^{n-1}B_n(x + j/k)$.  So, with $x = 1/6$ and $k = 2$ we have
 \[
 B_{2n+1}(1/6) = \frac{2^{2n}+1}{2^{2n}}B_{2n+1}(1/3)
 \]
 
 \begin{proof}[Proof of Corollary \ref{cor1.3}]
Let $p \equiv 1 \Mod 3$.  Then from Lemma \ref{computation} $p$ is 1-exceptional for $m = 3$ if and only if 
\[
\frac{B_p(1/3) - 2^{p}B_{p}(1/6)}{p^2} = -\frac{(1 + 2^p)B_{p}(1/3)}{p^2} =   \left(\frac{1 + 2^p}{3^p}\right)\frac{G_{p-1}}{p} \equiv 0 \Mod p.
\]
 The result now follows from Theorem \ref{mainthrm}.
 \end{proof}
 
 \begin{remark}
 From the proof of Corollary \ref{cor1.3} we also have that $G_{p-1}\equiv 0 \Mod p$ for all primes $p \equiv 1 \Mod 3$.
 \end{remark}
 
 \section{Some further questions}

Dummit, Ford, Kisilevsky and Sands conjecture in \cite{dfks} that given a fixed imaginary quadratic field $K$, there are infinitely many primes such that $\lambda_p(K) > 1$.  We can now restate this conjecture in the case of $K = \Q(i)$ and $K = \Q(\sqrt{-3})$ in a way that may be of interest to those who study Euler and Glaisher numbers, as well as Gauss factorials:
\begin{conjecture}
There are infinitely many primes $p \equiv 1 \Mod 3$ such that $G_{p-1} \equiv 0 \Mod{p^2}$.  Equivalently, there are infinitely many primes $p \equiv 1 \Mod 3$ such that $p$ is 1-exceptional for $m = 3$.
\end{conjecture}

\begin{conjecture}
There are infinitely many primes $p \equiv 1 \Mod 4$ such that $E_{p-1} \equiv 0 \Mod{p^2}$.  Equivalently, there are infinitely many primes $p \equiv 1 \Mod 4$ such that $p$ is 1-exceptional for $m = 4$.
\end{conjecture}

\section{Acknowledgments}

The author wishes to thank Nancy Childress, John Cosgrave and Karl Dilcher without whom this work would not be possible.  He also wishes to thank Neil Sloane for OEIS:A239902 \cite{sloane} which is where the author found \cite{cd1} and \cite{cd}.  The author is also grateful to the referee for helpful suggestions, in particular, comments which led to the proof of Theorem \ref{maxclassnumber}.



\end{document}